\documentclass[preprint,10pt]{elsarticle}
\usepackage[latin1]{inputenc}
\usepackage{graphicx}
\usepackage{amssymb}
\usepackage{amsmath}
\usepackage[nodots]{numcompress}

\def\+{\oplus}

\newcommand{\G}{\Gamma}

\newcommand{\R}{{\mathbb R}}

\newcommand{\N}{{\mathbb N}}

\newcommand{\cF}{{\mathcal F}}

\newcommand{\cI}{{\mathcal I}}

\newcommand{\cB}{{\mathcal B}}

\newcommand{\cS}{{\mathcal S}}

\renewcommand{\phi}{\varphi}
\renewcommand{\a}{{\alpha}}
\renewcommand{\b}{{\beta}}
\newcommand{\g}{{\gamma}}
\renewcommand{\d}{{\delta}}
\newcommand{\e}{\epsilon}

\newcommand{\pd}{\partial}
\newcommand{\Dx}{{\Delta x}}

\def\squareforqed{\hbox{\rlap{$\sqcap$}$\sqcup$}}
\def\qed{\ifmmode\else\unskip\quad\fi\squareforqed}
\def\smartqed{\def\qed{\ifmmode\squareforqed\else{\unskip\nobreak\hfil
\penalty50\hskip1em\null\nobreak\hfil\squareforqed
\parfillskip=0pt\finalhyphendemerits=0\endgraf}\fi}}

\newenvironment{proof}[1][Proof]{\textbf{#1.} }{\ \rule{0.5em}{0.5em}}
\newtheorem{remark}{\textbf{Remark}}[section]

\newtheorem{theorem}{\textbf{Theorem}}[section]

\newtheorem{proposition}{\textbf{Proposition}}[section]
\newtheorem{definition}{\textbf{Definition}}[section]

\numberwithin{equation}{section}
\begin{document}

\begin{frontmatter}

\title{An approximation scheme for an  Hamilton-Jacobi equation defined on a network}
\author{Fabio Camilli}
\address{Dipartimento di Scienze di Base e Applicate per l'Ingegneria,  ``Sapienza" Universit{\`a}  di Roma,
 00161 Roma, Italy } \ead{camilli@dmmm.uniroma1.it}
\author{Adriano Festa}
\ead{festa@mat.uniroma1.it}
\address{Dipartimento di Matematica,  ``Sapienza" Universit{\`a}  di Roma,
 00185 Roma, Italy}
\author{Dirk Schieborn}
\ead{Dirk@schieborn.de}
\address{Eberhard-Karls University,  T\"ubingen, Germany  }
\date{\today}


\begin{abstract}
In this paper we study an approximation scheme for an Hamilton-Jacobi equation
of Eikonal type defined on a network. We introduce an appropriate notion of
viscosity solution for this class of equations (see \cite{sc}) and we prove
that an approximation scheme of semi-Lagrangian type
converges to the unique solution of the problem.
\end{abstract}

\end{frontmatter}



\section{Introduction}
There is an increasing interest in the study of linear and nonlinear PDEs defined on networks
since they naturally arise in several applications (internet, vehicular traffic,  social networks, email exchange, disease transmission, etc.)
While a theory of linear  PDEs on networks is fairly complete (see \cite{n}, \cite{pb}), the study  of nonlinear problems
is very recent (\cite{gp}) and, concerning  Hamilton-Jacobi equations and control problems on networks,
is still at the beginning (see \cite{acct},  \cite{imz}, \cite{sc}).\par
It is well known that   Hamilton-Jacobi equations in general do   not admit regular
solutions and the correct notion of weak solution is the viscosity solution one. Hence all the three papers
concerning Hamilton-Jacobi equations aim to extend the concept of viscosity solution to the case of network and, in particular,
 to find the correct transition condition at the internal vertices.
 But, since the papers are motivated by different model problems and therefore they  make  different assumptions
on the Hamiltonian at the vertices,   the resulting  definitions of viscosity solution are quite different, even if all of them  give  existence and uniqueness of the solution.\par
The definition of viscosity solution introduced in \cite{sc} satisfies a stability property with respect to the uniform convergence. In this paper, we take advantage of this property to prove the convergence of  a numerical scheme for Hamilton-Jacobi equations on a network. For sake of simplicity we consider an Hamiltonian of Eikonal type, i.e. $H(x,p)=|p|-f(x)$, with a   Dirichlet boundary condition,  but the results can be extended to a more general class of  Hamiltonians   and also to other  boundary conditions.\par
Following \cite{ff}, we introduce a scheme of semi-Lagrangian type   by discretizing with respect to the time  the representation formula
for the solution of the Dirichlet problem.
We prove the well posed-ness of the discrete problem introducing  an appropriate discrete  transition condition and the   convergence
of the scheme to the solution of the continuous problem. It is worth noticing that the proof can be adapted  to prove convergence
of other approximation schemes, for example based on finite difference approximation.\par
In the second part  of the paper we study a fully discrete scheme which gives a finite-dimensional problem.
The scheme is obtained via a finite element discretization of semi-discrete
problem.  Also for this step of the discretization procedure we prove  the well posed-ness of the discrete problem and the convergence of the scheme to the unique solution of the continuous problem.  It is important to observe that the scheme
not only computes the solution of the Eikonal equations, but  it also  produces  an
approximation of the shortest paths to the boundary.\par
We also  discuss some issues concerning the implementation of the algorithm and  we  present some  numerical examples.

\section{Assumptions and preliminary results}
We give the definition of   graph suitable for our problem. We will also use the equivalent terminology
of  topological network (see \cite{lu}).
\begin{definition}
Let $V=\{v_i,\,i\in I\}$ be a finite collection of different points in $\R^N$ and
let $\{\pi_j,\,j\in J\}$ be a finite collection of differentiable, non self-intersecting curves in $\R^N$ given by
\[\pi_j:[0,l_j]\to\R^N,\, l_j>0,\,j\in J.\]
Set $e_j:=\pi_j((0,l_j))$, $\bar e_j:=\pi_j([0,l_j])$, and $E:=\{e_j:\, j\in J\}$. Furthermore assume that
\begin{itemize}
  \item[i)] $\pi_j(0), \pi_j(l_j)\in V$ for all $j\in J$,
  \item[ii)]$\#(\bar e_j\cap V)=2$ for all $j\in J$,
  \item[iii)] $\bar e_j\cap \bar e_k\subset  V$, and $\#(\bar e_j\cap \bar e_k)\le 1$ for all $j,k \in J$, $j\neq k$.
  \item[iv)] For all $v, w \in V$ there is a path with end-points  $v $ and $w$ (i.e.  a sequence of edges $\{e_j\}_{j=1}^N$ such that
  $\#(\bar e_j\cap \bar e_{j+1})=1$ and  $v\in \bar e_1$, $w\in \bar e_N$).
\end{itemize}
Then $\bar \G:=\bigcup_{j\in J}\bar e_j\subset \R^N$ is called a (finite)  \emph{topological network} in $\R^N$.
\end{definition}
For $i\in I$ we set $Inc_i:=\{j\in J:\,e_j \,\text{is incident to}\,v_i\}$.
Given a  nonempty set $I_B\subset I$, we define  $\partial \G:=\{v_i,\,i\in I_B\}$ (we always assume $i\in I_B$ whenever $\#(Inc_i)=1$ for some $i\in I$.) We set
 $I_T:=I\setminus I_B$ and  $\G:=\bar \G\setminus \partial\G$. \\
For any function $u:\bar \G\to\R$ and each $j\in J$ we denote by $u^j$ the restriction of $u$ to $\bar e_j$, i.e.
\[u^j:=u\circ \pi_j:[0,l_j]\to \R.\]
We say that $u$ is continuous in $\bar\G$ and write $u\in C(\bar\G)$ if $u$ is continuous with respect to the subspace topology of $\bar\G$. This means
that $u^j\in C([0,l_j])$ for any $j\in J$ and
\[u^j(\pi_j^{-1}(v_i))=u^k(\pi_k^{-1}(v_i)) \qquad\text{for any $i\in I$, $j,k\in Inc_i$.}\]
We define differentiation along an edge $e_j$ by
\[\pd_ju(x):=\pd_j u^j(\pi_j^{-1}(x))=\frac{\pd}{\pd x} u^j(\pi_j^{-1}(x)),\qquad\text{for  $x\in e_j$,}\]
and at a vertex $v_i$ by
\[\pd_ju(x):=\pd_j u^j(\pi_j^{-1}(x))=\frac{\pd}{\pd x} u^j(\pi_j^{-1}(x))\qquad\text{for $x=v_i$, $j\in Inc_i$.}\]
Observe that the parametrization of the arcs $e_j$ induces an orientation on the edges, which can be expressed
by the \emph{signed incidence matrix} $A=\{a_{ij}\}_{i,j\in J}$ with
\begin{equation}\label{1:1}
   a_{ij}:=\left\{
            \begin{array}{rl}
              1 & \hbox{if $v_i\in\bar e_j$ and $\pi_j(0)=v_i$,} \\
              -1 & \hbox{if $v_i\in\bar e_j$ and $\pi_j(l_j)=v_i$,} \\
              0 & \hbox{otherwise.}
            \end{array}
          \right.
\end{equation}
\begin{definition}\label{1:def2}
Let $\phi\in C(\G)$. \begin{itemize}
  \item[i)] Let $x\in e_j$, $j\in J$. We say that $\phi$ is differentiable at $x$, if $\phi^j$ is differentiable at $\pi_j^{-1}(x)$.
  \item[ii)] Let $x=v_i$, $i\in I_T$, $j,k\in Inc_i$, $j\neq k$. We say that $\phi$ is $(j,k)$-differentiable at $x$, if
\begin{equation}\label{1:2}
   a_{ij}\pd_j \phi_j(\pi_j^{-1}(x))+a_{ik}\pd_k \phi_k(\pi_k^{-1}(x))=0,
\end{equation}
where $(a_{ij})$ as in \eqref{1:1}.
\end{itemize}
\end{definition}
\begin{remark}
Condition \eqref{1:2} demands that the derivatives in the direction of the incident edges $j$ and $k$ at  the vertex $v_i$ coincide,
 taking into account the orientation of the edges.
\end{remark}
We consider the eikonal equation
\begin{equation}\label{eik}
    |\pd u| -f(x)=0,\qquad x\in \G.
\end{equation}
where   $f\in C^0(\bar \G)$, i.e.  $f(x)=f^j(\pi_j^{-1}(x))$ for $x\in \bar e_j$,
  $f^j\in C^0([0,l_j])$, and   $f^j(\pi_j^{-1}(v_i))=f^k(\pi_k^{-1}(v_i))$ for any $i\in I$, $j,k\in Inc_i$. Moreover
we  assume that
\begin{equation}\label{cost}
   f(x)\ge\eta>0\qquad x\in \G
\end{equation}
\begin{definition}\label{1:def3}\hfill\\
A function $u\in\text{USC}(\bar\G)$ is called a (viscosity) subsolution of \eqref{eik} in $\G$ if the following holds:
\begin{itemize}
  \item[i)] For any  $x\in e_j$, $j\in J$, and for any $\phi\in C(\G)$ which is differentiable at $x$ and for which
  $u-\phi$ attains a local maximum at $x$, we have
 \[|\pd_j\phi(x)|-f(x):=|\pd_j \phi_j(\pi_j^{-1}(x))|-f^j(\pi_j^{-1}(x))\le 0.\]
  \item [ii)] For any $x=v_i$, $i\in I_T$, and for any $\phi$ which is $(j,k)$-differentiable at $x$ and for which
  $u-\phi$ attains a local maximum at $x$, we have
 \[|\pd_j\phi(x)|-f(x)\le 0.\]
\end{itemize}
 A function $u\in\text{LSC}(\bar\G)$ is called a (viscosity) supersolution of \eqref{eik} in $\G$ if the following holds:
\begin{itemize}
  \item[i)] For any  $x\in e_j$, $j\in J$,  and for any $\phi\in C(\G)$ which is differentiable at $x$ and for which $u-\phi$ attains a
  local minimum at $x$, we have
 \[|\pd_j\phi(x)|-f(x)\ge 0.\]
  \item [ii)] For any $x=v_i$, $i\in I_T$, $j\in Inc_i$, there exists $k\in Inc_i$, $k\neq j$, (which we will call $i$-feasible for $j$ at $x$) such that for any $\phi\in C(\G)$ which is
  $(j,k)$-differentiable
  at $x$ and for which $u-\phi$ attains a local minimum at $x$, we have
 \[|\pd_j\phi(x)|-f(x)\ge 0.\]
\end{itemize}
A continuous function $u\in C(\G)$ is called a (viscosity) solution of  \eqref{eik} if it is both a viscosity subsolution and a viscosity supersolution.
\end{definition}

\begin{remark}\label{1:r2}
 Let $i\in I_T$ and  $\phi\in C(\G)$  be $(j,k)$-differentiable at $x=v_i$. Then
\begin{align*}\label{1:r21}
|\pd_j\phi(x)|-f(x)&=|\pd_j \phi_j(\pi_j^{-1}(x))|-f^j(\pi_j^{-1}(x))\\
&=|\pm\pd_j \phi_k(\pi_k^{-1}(x))|-f^k(\pi_k^{-1}(x)) =|\pd_k\phi(x)|-f(x),
\end{align*}
hence in  the subsolution and supersolution condition at the vertices, it is indifferent to require the condition for $j$ or for $k$. \\
\end{remark}
We give a representation formula for the  solution of \eqref{eik} completed with the   Dirichlet boundary condition
\begin{equation}\label{bc}
    u(x)=g(x)\qquad x\in \pd \G
\end{equation}
We define a distance-like function $S:\bar\G\times\bar\G\to [0,\infty)$  by
\[
   S(x,y):=\inf\left\{\int_0^tf(\g(s))ds:\, t>0,\,\g\in B^t_{x,y}\right\}
\]
where
\begin{itemize}
\item[i)] $\g:[0,t]\to \G$ is a piecewise differentiable path in the sense that there are
$t_0:=0<t_1<\dots<t_{n+1}:=t$ such that for any $m=0,\dots,n$, we have $\g([t_m,t_{m+1}])\subset \bar e_{j_m}$ for some
$j_m\in J$, $\pi_{j_m}^{-1}\circ \g\in C^1(t_m,t_{m+1})$, and
\[|\dot\g(s)|=\left|\frac{d}{ds}(\pi_{j_m}^{-1}\circ  \g)(s) \right|=1.\]
  \item[ii)]  $B^t_{x,y}$ is the set of all such paths with $\g(0)=x$, $\g(t)=y$.
\end{itemize}
If $f(x)\equiv1$, then  $S(x,y)$ coincides with the path distance $d(x,y)$ on the graph, i.e. the distance given by the length
of shortest arc in $\bar \G$ connecting $y$ to $x$. The following result is in the spirit of the corresponding results in $\R^N$
in \cite{cs}, \cite{fs}, \cite{im} (for the proof, see \cite[Proposition 6.1]{sc})
\begin{theorem}\label{uniqvisco}
Let $g:\bar\G\to \R$ be a continuous function satisfying
\begin{equation}\label{comp}
    g(x)-g(y)\le S(y,x)\qquad\text{for any $x$, $y\in \partial \G$}.
\end{equation}
Then the unique viscosity solution of \eqref{eik}--\eqref{bc} is given by
 \begin{equation}\label{representation}
 u(x):=\min\{g(y)+S(y,x):\, y\in   \partial \G\}.
 \end{equation}
\end{theorem}

\begin{remark}
It is worthwhile to observe that if supersolutions were defined similarly
to subsolutions, then the supersolution condition  could not be    satisfied by \eqref{representation}.
Consider the network $\G=\cup_{i=1}^3 e_i\subset\R^2$, where $e_1=\{0\}\times[0,1/2]$, $e_2=\{0\}\times[-1,0]$, $e_3=[0,1]\times \{0\}$
and  the equation $|\pd u|-1=0$ with zero boundary conditions at the vertices $v_1=(0,1/2)$, $v_2=(0,-1)$, $v_3=(1,0)$.
Then the distance  solution, see Theorem \ref{uniqvisco}, is given by $u(x)=\inf\{d(y,x):\,y\in \pd \G\}$ where $d$ is the path distance on the network.  The restriction of $u$  to $e_2\cup e_3$
has a local minimum at the vertex $v_0=(0,0)$. Hence   if $\phi$ is a constant function,  $u-\phi$ has a local
minimum at $v_0$ and therefore the supersolution condition is not satisfied for the couple $(e_2,e_3)$. Instead the arc $e_1$
is $v_0$-feasible; see the definition of supersolution, for both the arcs $e_2$ and $e_3$.
\end{remark}
\section{The approximation scheme}
We consider an approximation scheme of semi-Lagrangian type for the problem \eqref{eik}--\eqref{bc}.
\subsection{Semi-discretization in time}
Following the approach of \cite{ff} we construct an
approximation scheme for the equation \eqref{eik} by discretizing
the representation formula \eqref{representation}. We fix a discretization step $h>0$
 and we define a function $u_h:\bar\G\to\R$ by
\begin{equation}\label{11}
    u_h(x)=\inf\{\cF_h(\g^h)+g(y):\,\g^h\in B^h_{x,y},\,y\in\pd\G\}
\end{equation}
where $\cF_h(\g^h)=\sum_{m=0}^Mhf(\g^h_m)|q_m|$ and
\begin{itemize}
  \item[i)] An admissible trajectory $\g^h=\{\g^h_m\}_{m=1}^{M}\subset\G$ is a finite number of points $\g^h_m=\pi_{j_m}(t_m)\in \G$
 such that for any $m=0,\dots,M$, the arc $\widehat{\g^h_m\g^h_{m+1}} \subset \bar e_{j_m}$ for some $j_m\in J$ and $|q_m|:=|\frac{t_{m+1}-t_m}{h}|\le 1$
  \item[ii)]  $B^h_{x,y}$ is the set of all such paths with $\g^h_0=x$, $\g^h_{M}=y$.
\end{itemize}
\begin{remark}\label{disc-cont}
Given  $\g^h\in B^h_{x,y}$, we  define a   continuous path, still denoted by $\g^h$,  in $B_{x,y}$
by setting $\g^h(s)=\pi_{j_m}(t_m+\frac{(s-mh)}{h}(t_{m+1}-t_m))$ for $s\in [mh,(m+1)h]$ if $\widehat{\g^h_m\g^h_{m+1}} \subset \bar e_{j_m}$.
Then, recalling formula \eqref{representation} we approximate
 \[\int_0^{Mh}f(\g(s))|\dot\g(s))|ds=\sum_{m=1}^M\int_{(m-1)h}^{mh}f(\g(s))|q_m|ds\approx\sum_{m=1}^Mhf(\g^h_m)|q_m|\]
which shows that \eqref{11} is an approximation of \eqref{representation}. In the continuous case it is always possible
to assume by reparametrization that $|\dot\g(s)|=1$. In the discrete one we consider instead  velocities    in the interval $[-1,1]$, since otherwise
near the vertices the discrete dynamics can move only in one direction.
\end{remark}
Let $\cB(\G)$ be the space of the bounded functions on the network.
We show that the function $u_h$ can be characterized as the unique  solution of
the semi-discrete problem
\begin{equation}\label{HJh}
u_h(x)=S(h,x,u_h)
\end{equation}
where the scheme $S:\R^+\times\bar\G\times\cB(\G)\to\R$ is defined by
\begin{align}
    &S(h,x,\phi)=\inf_{q\in[-1,1]:\,x_{hq}\in \bar e_j}\{\phi( x_{hq} )+hf(x)|q|\}\label{12}\\
    & \phantom{aaaaaaaaaaaaaaaaaaaaaaaaaaaaaaaaaaaaaaa}\text{if  $x=\pi_j(t)\in e_j$} \nonumber\\
    &S(h,x,\phi)= \inf_{k\in Inc_i}\left[\inf_{ q\in [-1,1]:\,x_{hq}\in \bar e_k }\{\phi(x_{hq})+hf(x)|q|\}\right]\label{13}\\
    & \phantom{aaaaaaaaaaaaaaaaaaaaaaaaaaaaaaaaaaaaaaa}\text{if  $x=v_i$, $i\in I_T$}\nonumber\\
    &S(h,x,\phi)= g(x)\label{14} \\
    & \phantom{aaaaaaaaaaaaaaaaaaaaaaaaaaaaaaaaaaaaaaa}\text{if  $x\in \pd \G$}\nonumber
\end{align}
where, for $x=\pi_j(t)$, we define  $x_{hq}:=\pi_j(t-hq)$.\par
\begin{proposition}\label{discrete_problem}
Assume that
\begin{equation}\label{comph}
    g(x)\le \inf\{\cF_h(\g)+g(y):\,\g\in B^h_{x,y},\,y\in\pd\G\}\qquad \text{for any $x\in \pd \G$.}
\end{equation}
Then $u_h$ is the unique solution of \eqref{HJh}. Moreover $u_h$ is Lipschitz continuous uniformly in $h$, i.e.
\begin{equation}\label{16}
     |u_h(x_1)-u_h(x_2)|\le Cd(x_1,x_2) \qquad \text{for any $x_1,x_2\in\bar \G$}
\end{equation}
\end{proposition}
\begin{proof}
Let $u_1$, $u_2$ be two bounded solutions of \eqref{HJh} and set
$w_i(x)=1-e^{-u_i(x)}$, for $i=1,2$. Then
$w_i$ satisfies
\begin{equation}\label{uniq:1}
 w_i(x) =  \bar S(h,x,w_i)
\end{equation}
where
\begin{align*}
    &\bar S(h,x,\phi)=\inf_{q\in [-1,1]:\,x_{hq}\in \bar e_j}\{e^{-hf(x)|q|}\phi( x_{hq} )+1-e^{-hf(x)|q|}\} \\
    & \phantom{aaaaaaaaaaaaaaaaaaaaaaaaaaaaaaaaaaaaaaaaaaaa}\text{if  $x=\pi_j(t)\in e_j$} \\
    &\bar S(h,x,\phi)= \inf_{k\in Inc_i}\left[\inf_{ q\in [-1,1]:\,x_{hq}\in \bar e_k }\{e^{-hf(x)|q|}\phi(x_{hq})+1-e^{-hf(x)|q|}\}\right]\\
    & \phantom{aaaaaaaaaaaaaaaaaaaaaaaaaaaaaaaaaaaaaaaaaaaa}\text{if  $x=v_i$, $i\in I_T$}\\
    &\bar S(h,x,\phi)= 1-e^{-g(x)}\\
    & \phantom{aaaaaaaaaaaaaaaaaaaaaaaaaaaaaaaaaaaaaaaaaaaa}\text{if  $x\in \pd \G$}
\end{align*}
where, for $x=\pi_j(t)$,  $x_{hq}:=\pi_j(t-hq)$. In fact, for any $q\in [-1,1]$ such that $x_{hq}\in \bar e_j$, we have
\begin{align*}
    w_i(x)=1- e^{-u_i(x)}\le 1-e^{-u_i( x_{hq} )-hf(x)|q|}=1-e^{-u_i( x_{hq} )}e^{-hf(x)|q|}=\\
    (1-e^{-u_i( x_{hq} )})\,e^{-hf(x)|q|}+1-e^{-hf(x)|q|}=e^{-hf(x)|q|}w_i(x_{hq})+1-e^{-hf(x)|q|}
\end{align*}
and the first equation in \eqref{uniq:1} follows taking the infimum with respect to $q$. We proceed similarly for the
other two equations.\par
We have that
\[\sup_{\G}|\bar S(h,x,w_1(x))-\bar S(h,x,w_2(x))|\le \beta \sup_{\G}|w_1(x)-w_2(x)|\]
with $\beta=e^{-h\eta}<1$, see \eqref{cost}.
Since $\bar S$ is a contraction, we conclude  that for  $h>0$  there exists at most one
bounded solution of \eqref{uniq:1} and therefore  of problem  \eqref{HJh}. \par
Now we show the function $u_h$ is a bounded solution of \eqref{12}--\eqref{14}.
It is always  possible to  assume, by adding a
constant, that $g\ge 0$. It follows that $u_h\ge 0$. Moreover it is easy to see that
\[u_h(x)\le \|f\|_\infty\sup_{x\in\G}d(x,\pd \G)+\sup_{x\in\pd \G}g(x).\]
To show \eqref{14}, observe that we have  $u_h(x)\neq g(x)$ for  $x\in  \pd \G$ if and only if there is some $z\in  \pd \G$
such that $g(x)>g(z)+\cF_h(\g^h)$ for some $\g^h\in B^h_{z,x}$ which gives  a contradiction to \eqref{comph}.\par
We consider \eqref{12} and we   first show the   ``$\le$''-inequality. For $x\in e_j$ and  for   $q\in [-1,1]$ such that  $x_{hq}\in \bar e_j$, let $y\in \pd \G$ and  $\g_1^h\in B^h_{x_{hq},y}$ be $\e$-optimal for $u_h(x_{hq})$. Define $\g^h=\{\g_i^h\}_{i=0}^1$ with $\g_0^h=x$, $\g_1^h=x_{hq}$. Hence
    $  \g^h_1\cup\g^h\in B^h_{x,y}$ (with $x_{hq}$ counted only one time in $ \g^h_1\cup\g^h $) and
\[u_h(x)\le g(y)+ \cF_h(\g^h\cup \g_1^h)\le  g(y)+ \cF_h(\g^h)+hf(x)|q|\le u_h(x_{hq})+\e +hf(x)|q|.\]
To show the reverse inequality, assume that   for some $x\in \G$,
\[u_h(x)\le  \inf_{q\in[-1,1]:\,x_{hq}\in \bar e_j}\{u_h( x_{hq} )+hf(x)|q|\}-\d.\]
for $\d>0$. Given $\e<\d$, let $y\in \pd \G$ and  $\g^h_{x,y}=\{\g^h_m\}_{m=0}^{M}\in B^h_{x,y}$ be $\e$-optimal for $x$.
By the inequality
\[
g(y)+\cF_h(\g^h_{xy})-\e\le u_h(x)\le u_h(x_{hq})+ hf(x)|q|-\d
\]
it is clear that if $y=x_{hq}$ for some $q\in [-1,1]$ we get a contradiction. Define $\g^h=\g^h_{x,y}\setminus \g^h$ where
$\g^h=\{\g_i^h\}_{i=0}^1$ with $\g_0^h=x$, $\g_1^h=x_{hq}$.
Since $\bar\g^h:=\g^h_{x,y}\setminus  \g^h\in B^h_{x_{hq},y}$ we have
\[
g(y)+\cF_h(\bar\g^h)=g(y)+\cF_h(\g^h_{x,y})-\cF_h(\g^h)\le u_h(x_{hq})+\e-\d
\]
a contradiction to the definition of $u_h$ and therefore \eqref{12}. The equation \eqref{13} can be proved in a similar way.\\
We finally show that the function $u_h$ is Lipschitz continuous in $\G$, uniformly in $h$.
Consider first the case of two points in the same arc, i.e. $x_1 $, $x_2\in \bar e_j$ for some $j\in J$. Given $\e>0$, denote  by $\g^h=
\{\gamma^h_m\}\in B^h_{x_1,x_2}$ by
\begin{equation}\label{11c}
\gamma^h_m=\left\{
             \begin{array}{ll}
               x_1, & \hbox{$m=0$;} \\
               z_m, & \hbox{$m=1,\dots,M-1$;} \\
              x_2, & \hbox{$m=M$.}
             \end{array}
           \right.
\end{equation}
where  $|\pi_j^{-1}(\g_m)- \pi_j^{-1}(\g_{m+1})|\le h$  for $m=0,\dots,M$.
Let   $y\in \pd \G$ and  $\g_1^h\in B^h_{x_1,y}$ be $\e$-optimal for $x_1$. Then
$\g^h_1\cup \g^h\in B^h_{x_2,y}$ and
\begin{align*}
u_h(x_2)\le g(y)+ \cF_h(\g_1^h\cup \g_2^h) \le g(y)+\cF_h(\g^h_1)+\cF_h(\g^h_2)\\
\le u_h(x_1)+ C\sum_{m=0}^M h |\pi_j(t_{m+1}-\pi_j(t_m)|+\e
\le u_h(x_1)+Cd(x_1,x_2)+2\e
\end{align*}
Exchanging the role of $x_1$ and $x_2$   we get
\begin{equation}\label{16b}
    |u_h(x_1)-u_h(x_2)|\le Cd(x_1,x_2)
\end{equation}
If  $x_1, x_2\in \G$, let $\g$ be such that $\int_0^T|\dot \g(s)|ds\le d(x_1,x_2)+\e$ and $\{e_{j_m}\}_{m=1}^M\subset J$ such that $\g([0,T])\subset\cup_{m=1}^M e_{j_m}$.
For each one of the couples $(x_1, v_{j_1})$, $(v_{j_m}, v_{j_{m+1}})$ for $m=1,\dots,M$ and $(v_{j_M}, x_2)$   define a trajectory $\g_m^h$ as in \eqref{11c}.  Then define   $\g ^h\in B^h_{x_1,x_2}$ by
\[
\gamma^h =\left\{
             \begin{array}{ll}
               x_1, & \hbox{$k=0$;} \\[3pt]
               \g_k^h & \hbox{$k=\sum_{i=1}^m M_{i-1} ,\dots,\sum_{i=1}^m M_{i-1}+M_m-1$;} \\[3pt]
               x_2, & \hbox{$m=\bar M$.}
             \end{array}
           \right.
\]
where $\bar M=\sum_{i=0}^{M+1} M_{i}$.
For  $t_k=\pi_{j_m}^{-1}(\g^h_k)$, $k=\sum_{i=1}^m M_{i-1} ,\dots,\sum_{i=1}^m M_{i-1}+M_m-1$, then we have $t_{k+1}-t_k=hq_k$ with $|q_k|\le 1$.
Let   $y\in \pd \G$ and  $\g^h_1\in B^h_{x_1,y}$ be $\e$-optimal for $x_1$. Then
$\g^h_1\cup \g^h\in B^h_{x_2,y}$ and
\begin{align*}
u_h(x_2)\le g(y)+ \cF_h(\g^h_1\cup \g^h_2) \le g(y)+\cF_h (\g^h_2)+\cF_h (\g^h_2)\\
\le u_h(x_1)+\sum_{k=0}^{\bar M} h|q_k|f(\g_k^h)+\e \le u_h(x_1)+Cd(x_1,x_2)+2\e.
\end{align*}
Exchanging the role of $x_1$ and $x_2$   we get \eqref{16b}
\end{proof}
\begin{remark}
By  Remark \ref{disc-cont} and the continuity of $f$, assumption \eqref{comp} implies
\[
    g(x)\le \inf\{\cF_h(\g)+g(y):\,\g\in B^h_{x,y},\,y\in\pd\G\}+Ch\qquad \text{for any $x,y\in \pd \G$.}
\]
Moreover, if $g\equiv 0$ on $\pd \G$, the condition  \eqref{comph} is satisfied since $\cF_h(\g^h)\ge 0$ for any $\g^h$.
\end{remark}
\begin{theorem}\label{convh}
Assume \eqref{comph} for any $h>0$ and \eqref{comp}. Then for $h\to 0$,  the solution $u_h$ of \eqref{HJh} converges
 uniformly to the unique solution $u$ of \eqref{eik}-\eqref{bc}.
\end{theorem}
\begin{proof}
we first observe that \eqref{eik} can be written in equivalent form as
\[\sup_{q\in [-1,1]}\{-q\,\pd u(x) -f(x)|q|\}=0\]
By \eqref{16}, $u_h$ converges, up to a subsequence, to a Lipschitz continuous function $u$. We show that $u$
is a solution of \eqref{eik} at $x\in \G$. We will consider the case $x=v_i\in I_T$, as otherwise the argument is standard   (see f.e. \cite[Th.VI.1.1]{bcd}).\par
To show that $u$ is a  \emph{subsolution},
choose any $j,k\in Inc_i$, $j\neq k$, along with
an   $(j,k)$-test function $\phi$   of $u$ at $x$.
Observe that it is not restrictive to consider $x$ to be a strict maximum point for
$u-\phi$, since we otherwise consider the auxiliary function $\phi_\d(y):=\phi (y)+\d d(x,y)^2$ for $\d>0$ with $\pd_m ( d(x,\cdot)^2)(\pi_m^{-1}(x))=0$ for    $m=j$ and  $m=k$.
Then there exists $r>0$ such that $u-\phi$ attains a strict local
maximum w.r.t. $\bar B_r(x)$ at $x$, where $B_r(x):=\{y\in\G:\, d(x,y)<r\}$. Moreover $x$ is a strict maximum point
for $u-\phi$ also in $\bar B:=\bar B_r(x)\cap (\bar e_j \cup\bar e_k)$. Now choose a sequence $\omega_h\to 0$ for $h\to 0$ with
\begin{equation}\label{17b}
   \sup_{\G}|u(x)-u_h(x)|\le \omega_h
\end{equation}
and let $y_h$ be a maximum point for $u_h-\phi $ in $\bar B$. Up to  a subsequence, $y_h\to z\in \bar B$. Moreover,
\[u(x)-\phi (x)-\omega_h\le u_h(x)-\phi (x)\le u_h(y_h)-\phi (y_h)\le u(y_h)-\phi (y_h)+\omega_h.\]
For $h\to 0$, we get
$u(x)-\phi (x)\le u(z)-\phi (z).$
As $x$ is a strict maximum point, we conclude $x=z$. Invoking
\[u(x)+\phi(y_h)-\phi (x)-\omega_h\le u_h(y_h)\le u(y_h)+\omega_h\]
we altogether get
\begin{equation}\label{18}
\lim_{h\to 0} y_h= x, \quad    \lim_{h\to 0}u_h(y_h)=u(x)
\end{equation}
We distinguish two cases:\\
\emph{Case 1: $y_h\neq x$. } Then $y_h\in e_m$ with either $m=j$ or $m=k$. Since $u_h-\phi$ attains a maximum at $y_h$,
then for $y_h=\pi_m(t_h)$ and  $y_{hq}=\pi_m(t_h-hq)\in \bar e_m$
\[u_h(y_h)-\phi(y_h)\ge u_h(\pi^{-1}_m(y_{hq}))-\phi(\pi^{-1}_m(y_{hq}))\]
and therefore
\begin{equation}\label{19}
   \sup_{q\in[-1,1]:\,y_{hq}\in \bar e_m}\left\{-\frac{\phi(\pi^{-1}_m(y_{hq}))-\phi(\pi^{-1}_m(y_h))}{h}-hf^m(y_h)|q|\right\}\le 0
\end{equation}
The set $\{q\in\R:\,\pi_m(t-hq)\in \bar e_m\}$
contains for $h$ small enough  either $[-1,0]$ if $a_{i,m}=1$ or  $[0,1]$ if $a_{i,m}=-1$.
Passing to the limit for $h\to 0$ in \eqref{19},  since
$f^m(x)|q|=f^m(x)|-q|$  we get
\[\sup_{q\in [-1,1]}\{q\,\pd_m \phi(x)-f(x)|q|\}\le 0.\]
\emph{Case 2: $y_h= x$. }
Since $u_h-\phi$ attains a maximum at $x$,
then for $x=\pi_j(t_h)$ and  $y_{hq}=\pi_j(t_h-hq)\in \bar e_j$
\[u_h(y_h)-\phi(y_h)\ge u(y_{hq})-\phi(y_{hq})\]
and therefore
\[
   \sup_{q\in [-1,1]:\,y_{hq}\in \bar e_j}\left\{-\frac{\phi^j_h(y_{hq})-\phi_h^j(y_h)}{h}- hf^j(y_h)|q|\right\}\le 0
\]
The set $\{q\in\R:\,\pi_j(t-hq)\in \bar e_j\}$
contains for $h$ small enough  either $[-1,0]$ if $a_{i,j}=1$ or  $[0,1]$ if $a_{i,j}=-1$
and passing to the limit for $h\to 0$  we conclude as in the previous case  that
\[
 \sup_{q\in [-1,1]}\{q\,\pd_j \phi(x)-f(x)|q|\}\le 0.
\]
To show  that $u$ is a \emph{supersolution}, we
 assume by contradiction that  there exists $j\in Inc_i$ such that for any $k\in Inc_i$, $k\neq j$, there exists
 a   $(j,k)$-test function $\phi_k$ of $u$ at $x$ for which
 \begin{equation}\label{2:prop25}
    \sup_{q\in [-1,1]}\{q\,\pd_j \phi_k(x)-f(x)|q|\} <0.
 \end{equation}
By adding a quadratic function of the form $-\a_kd(x,y)^2$ to the function $\phi_k$ we may assume that
there exists $r>0$ such that $u-\phi_k$ attains a strict minimum in $\bar B_r(x)$ at $x$. Observe that $x$ is a strict minimum point
of $u-\phi_k$ also in $\bar B_k:=\bar B_r(x)\cap (\bar e_j \cup\bar e_k)$.\\
Since for any $h$, there exists $k_h$ such that
\[u_h^j(v_i)=\inf_{{q\in[-1,1]:\,\pi_{k_h}(t-hq)\in \bar e_{k_h}}}\{u^{k_h}_h(\pi_{k_h}(t-hq))+hf^{k_h}(v_i)|q|\}\]
we may assume, up to a subsequence, that there exists $k\in Inc_i$ such that $k_h=k$ for any $h>0$.\\
 Let $y_h$ be a minimum point of $u_h-\phi_k$ in $\bar B_k$ and let $\omega_h$ be as in \eqref{17b}.
 As in  the subsolution case, we   prove that \eqref{18} holds.
If $y_h\neq x$, we have
for $y_h=\pi_m(t_h)$ and  $t_h-hq\in \bar e_m$
\[u_h(y_h)-\phi(y_h)\le u(\pi_m(t_h-hq))-\phi(\pi_m(t_h-hq))\]
and therefore
\begin{equation*}
   \sup_{q\in[-1,1]:\,\pi_m(t-hq)\in \bar e_m}\left\{-\frac{\phi^m_h(\pi_m(t_h-hq))-\phi_h^m(y_h)}{h}-hf^m(y_h)|q|\right\}\ge 0
\end{equation*}
for  either $m=j$ or $m=k$.
If $y_n= x$, we get
\begin{equation*}
   \sup_{q\in[-1,1]:\,\pi_j(t-hq)\in \bar e_j}\left\{-\frac{\phi^j_h(\pi_j(t_h-hq))-\phi_h^j(y_h)}{h}-hf^j(x)|q|\right\}\ge 0
\end{equation*}
Arguing as in the subsolution case we get for $h\to 0$
\[
 \sup_{q\in [-1,1]}\{q\,\pd_j \phi(x)-f(x)|q|\}\ge 0.
\]
which is a contradiction to \eqref{2:prop25}.\par
We conclude the proof by observing that the uniqueness of the solution to \eqref{eik} implies that any convergent subsequence $u_h$ must converge to the unique  solution   $u$ of \eqref{eik}-\eqref{bc} and therefore    the uniform convergence of all the sequence $u_h$ to $u$.
\end{proof}
\subsection{Fully discretization in space}
In this section we introduce a FEM like discretization of
\eqref{HJh} yielding a fully discrete scheme.
For any  $j\in J$, given  $\Dx^j>0$ we consider  a finite partition
\[P^j=\{t^j_1=0<\dots<t^j_m<\dots<t^j_{M_j}=l_j\}\] of the interval $[0,l_j]$
such that $|P^j|=\max_{1,\dots,M_j}(t^j_m-t^j_{m-1})\le \Dx^j$. We set
 \begin{equation}\label{k}
 \Dx=\max_{j\in J} \Dx^j, \qquad M=\sum_{j\in J} M_j
 \end{equation}
The partition $P^j$ induces a partition of the arc $\bar e_j$ given by the points
\[x_m^j=\pi_j(t_m^j), \qquad  m=1,\dots,M_j.\]
and  we set  $X_\Dx=\cup_{{j\in J}}\cup_{m=1}^{M_j}x_m^j$.\par
In each interval $[0,l_j]$ we consider a family   of basis functions
$\{\beta^j_m\}_{m=0}^{M_j}$  for  the space of continuous, piecewise linear functions in
the intervals of the partition $P^j$. Hence  $\beta^j_m$ are piecewise linear functions satisfying   $\beta^j_{m}(t_{k})=\delta_{mk}$ for $m,k\in \{1,\dots, M_j\}$
$0\leq \beta^j_m(t)\leq 1$, $\sum_{m=1}^{M_j}\beta^j_m(t)=1$ and for any $t\in [0,l_j]$ at most $2$
$\beta_m^j$'s are non-zero. We define $\bar \beta_j:\bar e_j\to \R$ by
\[\bar \b^j_m(x)=\beta^j_m(\pi_j^{-1}(x)).\]
 Given $W\in \R^M$ we denote by $\cI_\Dx[W]$ the interpolation operator
defined  on the arc $\bar e_j$ by
\[\cI^j_\Dx[W](x)=\sum_{m=1}^{M_j}\bar\b^j_m(x)W^j_m=\sum_{m=1}^{M_j}\b^j_m(\pi_j^{-1}(x))W^j_m\qquad x\in \bar e_j.\]
We consider the approximation scheme
\begin{equation}\label{HJhk}
    U=\cS(\Dx,h,U)
\end{equation}
where the scheme $\cS=\{\cS(\Dx,h,W)\}_{j\in J}$ is  given by
\begin{align}
    &\cS^j_m(\Dx,h,W)= \inf_{q\in[-1,1]:\,x^j_m(q)\in \bar e_j}\{\cI^j[W](x^j_m(q))+hf(x^j_m)|q|\}\label{12k} \\
    & \phantom{aaaaaaaaaaaaaaaaaaaaaaaaaaaaaaaaaaaaaaa}\quad \text{if  $x^j_m\in e_j$}\nonumber\\
    &\cS^j_m(\Dx,h,W)  =  \inf_{{q\in[-1,1]:\,x^k_m(q)\in \bar e_k}\atop {k\in Inc_i}}\{\cI^k[W](x^k_m(q))+hf(x^k_m)|q|\} \label{13k}\\
    & \phantom{aaaaaaaaaaaaaaaaaaaaaaaaaaaaaaaaaaaaaaa} \text{if  $x^j_m=v_i\in I_T$} \nonumber\\
&\cS^j_m(\Dx,h,W)= g(v_i) \label{14k} \\
& \phantom{aaaaaaaaaaaaaaaaaaaaaaaaaaaaaaaaaaaaaaa}\text{if  $x_m^j=v_i,\,i\in I_B$}\nonumber
\end{align}
for $x^j_m(q)=\pi^j(t^j_m-hq)$.
\begin{proposition}
For any $\Dx>0$ with $\Dx\le h/2$, there exists a unique solution $U\in\R^M$ to \eqref{12k}--\eqref{14k}. Moreover,
defined  $u_{h\Dx}(x)= \cI_\Dx[U]$, if $\Dx=o(h)$ for $h\to 0$, then $u_{h\Dx}$ converges  to the unique solution $u$ of \eqref{eik}-\eqref{bc} uniformly in $\G$.
\end{proposition}
\begin{proof}
We show the boundedness of a solution to \eqref{HJhk} by induction. For this purpose we number the nodes
$x_i$ such that $d(x_{i+1},\partial\G)\ge
d(x_{i},\partial\G)$ for all $i=1,\ldots, M$,
and claim that
\[ |U_i| \le \sup_{x\in\partial \G} |g(x)| + h (L_g +
M_f) + 2 M_f d(x_i,\partial \G). \]
For each $x_i$ with $d(x_i, \partial \G)\le h$ this estimate is
immediate.
Now assume the assertion is true for all $x_i$ with $i=1,\ldots,l-1$.
For $x_l\in \bar e_j$ by \eqref{HJhk} we obtain the inequality
\[ U_l \le h f(x_l)|q| +\cI^j[U](x^j_l(q))
\le hM_f +  \cI^j[U](x^j_l(q))\]
for any $q\in\R^n$ with $|q|\le 1$ and $x^j_l(q)\in \bar e_j$.
Choosing $q$ such that $d(x^j_l(q), \partial \G) = d(x_l, \partial \G)
- h$ and using $\Dx\le h/2$ we obtain that the value $\cI^j[U](x^j_l(q))$
only depends on nodes $x_{i_k}$ with $d(x_{i_k}, \partial \G) \le d(x_l,
\partial \G) - h/2$, thus $i_k<l$. Picking that node $x_{i_k}$ such that
$U_{i_k}$ becomes maximal, and using the induction
assumption we can conclude
\[ U_l \le M_fh + U_{i_k}
\le M_fh + \sup_{x\in\partial \G} |g(x)| + h (L_g +
M_f) + 2 M_f (d(x_i,\partial \G) - h/2)\]
i.e.\ the assertion.\par
To show existence of a unique solution $U$
we apply  the   transformation
\[ \widetilde U = 1-e^{-U} \]
to \eqref{HJhk}. Hence  $\widetilde U$ is a solution to
\begin{equation}\label{HJhkKr}
    \tilde U=\widetilde\cS(\Dx,h, U)
\end{equation}
where
\begin{align*}
 \qquad
    &\widetilde \cS^j_m(\Dx,h,\widetilde W)= \inf_{q\in[-1,1]:\,x^j_m(q)\in \bar e_j}\{e^{-hf(x^j_m)}\cI^j[\widetilde W](x^j_m(q))+1-e^{-hf(x^j_m)|q|}\}\\
    & \phantom{aaaaaaaaaaaaaaaaaaaaaaaaaaaaaaaaaaaaaaaaaaaa}\text{if  $x^j_m\in e_j$}\\
    &\widetilde \cS^j_m(\Dx,h,\widetilde W)  = \inf_{{q\in[-1,1]:\,x^k_m(q)\in \bar e_k}\atop{k\in Inc_i}}\{e^{-hf(x^k_m)}\cI^k[\widetilde W](x^k_m(q))+1-e^{-hf(x^k_m)|q|}\}\\
    & \phantom{aaaaaaaaaaaaaaaaaaaaaaaaaaaaaaaaaaaaaaaaaaaa}\text{if  $x^j_m=v_i\in I_T$} \\
 \qquad &\widetilde \cS^j_m(\Dx,h,\widetilde W)= 1-e^{-g(v_i)}\\
 & \phantom{aaaaaaaaaaaaaaaaaaaaaaaaaaaaaaaaaaaaaaaaaaa}\text{if  $x_m^j=v_i$, $i\in I_B$}
\end{align*}
As in the  proof of Proposition \ref{comph} we show that $\widetilde \cS$ is a contraction in $\R^M$
and we conclude that there exists a unique bounded solution to \eqref{HJhkKr} and therefore to \eqref{HJhk}.\par
To show the convergence of $u_{h\Dx}$ to $u$, we set $\tilde u_h=1-e^{-u_h}$, $\tilde u_{h\Dx}=1-e^{-u_{h \Dx}}$ and we estimate for $x\in\bar e_j$
\begin{equation}\label{HJhk1}
    |\tilde u_h(x)-\tilde u_{h\Dx}(x)|\le \big|\tilde u_h(x)-\cI^j[\tilde U^h](x)\big|+\big|\cI^j[\widetilde U^h](x)-\cI^j[\widetilde U](x)\big|\
\end{equation}
where $\tilde U^h$, $\tilde U$ are  the vectors of the values of $\tilde u_h$, $\tilde u_{h\Dx}$ at the nodes of the grid. By the Lipschitz
continuity and boundedness of $u_h$   we get
\begin{equation}\label{HJhk2}
    |\tilde u_h(x)-\cI^j[\tilde U^h](x)|\le C\Dx
\end{equation}
with $C$ independent of $h$. Moreover, by \eqref{uniq:1} and \eqref{HJhkKr} we get for $x_k=\pi_j^{-1}(t_k)\in e_j$,
 $x_{hq}:=\pi_j(t_k-hq)$ and since  $x^j_k(q)=x_{hq}$
\begin{equation}\label{HJhk3}
   \big|\tilde U^h_k-\tilde U_k\big|\le e^{-h
f(x_k)}|\tilde u_h(x_{hq})-\cI^j[\widetilde U](x^j_k(q))|\le e^{-h\eta}\|\tilde u_h-\tilde u_{h\Dx}\|_\infty
\end{equation}
where $\eta$ as in \eqref{cost}.
Substituting \eqref{HJhk2} and \eqref{HJhk3} in \eqref{HJhk1}
we get
\[\|\tilde u_h -\tilde u_{h\Dx}  \|_\infty\le \frac{C}{1-e^{-\eta h}}\Dx\]
and therefore, taking into account Theorem \ref{convh}, we have that if $\Dx=o(h)$ for $h\to 0$, then $u_{h\Dx}$
converges to $u$ uniformly on $\G$.
\end{proof}
\section{Implementation of the scheme and numerical tests}\label{Examples}
In this section we discuss  the numerical implementation of the  scheme described in the previous section and we present some
numerical examples. We remark again that the most interesting feature of our  approach is that it is  intrinsically one-dimensional, even if
the graph is embedded in $\R^N$. For this reason it does not present the typical curse of dimensionality issue   which is usually encountered in solving  Hamilton-Jacobi equations on $\R^N$.\par
The numerical implementation of   semi-Lagrangian schemes   has been extensively discussed in previous works (see for example the Appendix B in \cite{bcd}), hence the only  regard is due to vertices, where the information  could come from different arcs.
 We briefly  describe the logical structure of the algorithm we use to compute the solution.\\
Let   $A$ be the $m\times m$  incidence matrix defined  in \eqref{1:1}. We also define a  matrix $BC$ which contains the information on boundary vertices, in particular: $BC(\cdot,1)$ represents a boundary  vertex and $BC(\cdot,2)=$ the value of the Dirichlet  datum at that vertex.
The number of the edges is  at most $n=\frac{(m-1)m}{2}$ and, after having ordered the edges,  we   define the {\it auxiliary edges matrix} $B\in M^{3,n}$ where  the $i$-row contains the following information:
\begin{itemize}
	\item $B(i,1)=$ \#knot where the i-arc starts,
	\item $B(i,2)=$ \#knot where the i-arc ends,
	\item $B(i,3)=$ length of the discretized i-arc,
\end{itemize}
We choose the same  discretization step $\Delta x\equiv \Delta x_i$   for every edge, so that the approximated length of the  edge $i$ is $L_i=trunc(\frac{B(i,3)}{\Delta x })\in \N^+$ and we consider a finite partition
\begin{equation}
   P^i=\left\{t_0^i=0,t_1^i=\Delta x,t_2^i=2\Delta x, \cdots , t^i_{M_i-1}=(M_i-1)\Delta x,t^i_{M_i}=B(i,3)\right\}.
\end{equation} The matrix $C$, contains the grid points of the graph, i.e. for the  edge $i$
\begin{equation}
  C(i,j)=\pi_i(t^i_j)\quad j=0,\ldots,M_i
\end{equation}
Finally, we denote by  $U(i,j)$ the  the approximated solution at the point $C(i,j)$ point.   We solve
the problem using the following iteration
\vskip 8pt
  {\bf HJ-networks algorithm.}
\vskip 8pt
 \rule[-0.1cm]{6cm}{0.01cm}
\vskip 8pt
 1.   Initialize  \par
 \hspace{2cm}  $U=U_0$ ;  \par
\hspace{2cm} it=0; \par
\par
 2.  Until convergence, Do  \par
\par
3.   for i=0 to n  \par
            \par
 4.      \hspace{1cm} If there is an $s$ s.t. $B(i,1)=BC(s,1)$ \par

  5.      \hspace{1.5cm}  then $U(i,0)=BC(s,2)$;\par

   6.      \hspace{1.6cm}else \par
 7.      \hspace{1.5cm}   $U(i,0)=\min\left\{\min_{\left\{k|A(B(i,1),k)=1\right\}}\left\{I[U](C(k,\frac{h}{\Delta x}))\right\},\right.$\par
          \hspace{2.2cm}$\left.\min_{\left\{k|A(B(i,1)=-1\right\}}\left\{I[U](C(k,B(k,3)-\frac{h}{\Delta x}))\right\}\right\}+h f(C(i,j))$  \par

  8.   \hspace{1cm}   for $j=0$ to $B(i,3)-1$ \par

   9.     \hspace{2cm}   $U(i,j)=\min_{a\in\left[-1,1\right]}\left\{I[U](C(i,j+\frac{ah}{\Delta x}))\right\}+h f(C(i,j))$ \par

   10.      \hspace{1cm} If there is an $s$ s.t. $B(i,2)=BC(s,2)$\par

  11.      \hspace{1.5cm}  then $U(i,B(i,3))$=BC(s,2);\par

   12.      \hspace{1.6cm}else \par
   13.      \hspace{1.3cm}   $U(i,B(i,3))= \min\left\{\min_{\left\{k|A(B(i,2)=1\right\}}\left\{I[U](C(k,\frac{h}{\Delta x}))\right\},\right.$\par
   \hspace{2.2cm}$\left.\min_{\left\{k|A(B(i,2)=-1\right\}}\left\{I[U](C(k,B(k,3)-\frac{h}{\Delta x}))\right\}\right\}+h f(C(i,j))$   \par

14.      re-initialize vertex on $U$ \par

 15.    EndDo   \par
\medskip
\rule[-0.1cm]{6cm}{0.01cm}
\vskip 12pt

The interpolation $I[U](C(i,x))$ is  the usual linear interpolation, i.e., said $t(x)=trunc(x)$
\begin{equation}
\begin{split}
I[C](x)&= C(i,t(x))+\frac{(x-t(x))}{\Delta x}[C(i,t(x)+1)-C(i,t(x))] \\
I[U](C(i,x))&= U(i,t(x))+(I[C](x)-C(i,t(x)))\frac{U(i,t(x)+1)-U(i,t(x))}{C(i,t(x)+1)-C(i,t(x))}
\end{split}
\end{equation}

\begin{remark}
The order given to the edges, which is necessary to define the previous iteration,    brings some  additional problems that we have to consider:
\begin{itemize}
	\item At the end of each iteration of the method, the values of the solution at a   same vertex, which is contained in different arcs, could be different. Hence  we make  a re-initialization, choosing for every vertex the minimum of the  previous values.\par
	\item
It is also important that the initial guess $U_0$ of the solution we use to initialize the algorithm is greater than  the solution. In fact, if this condition is not satisfied,  for particular choices of the discretization step  the algorithm  could generate a non correct minimum.
\end{itemize}
\end{remark}

\begin{figure}[th]
\begin{center}
\includegraphics[height=6cm]{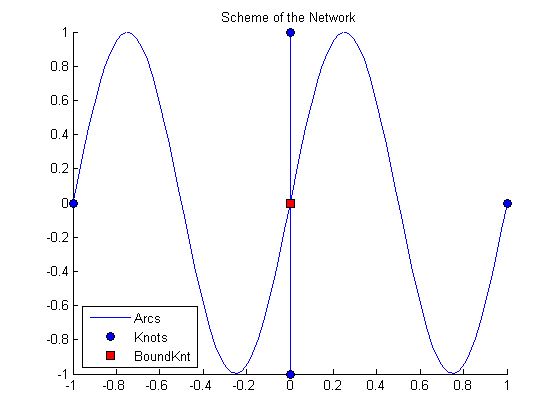}
\caption{Test 1, structure of the graph.} \label{scheme2}
\end{center}
\end{figure}

\begin{figure}[ht]
\begin{center}
\includegraphics[height=7cm]{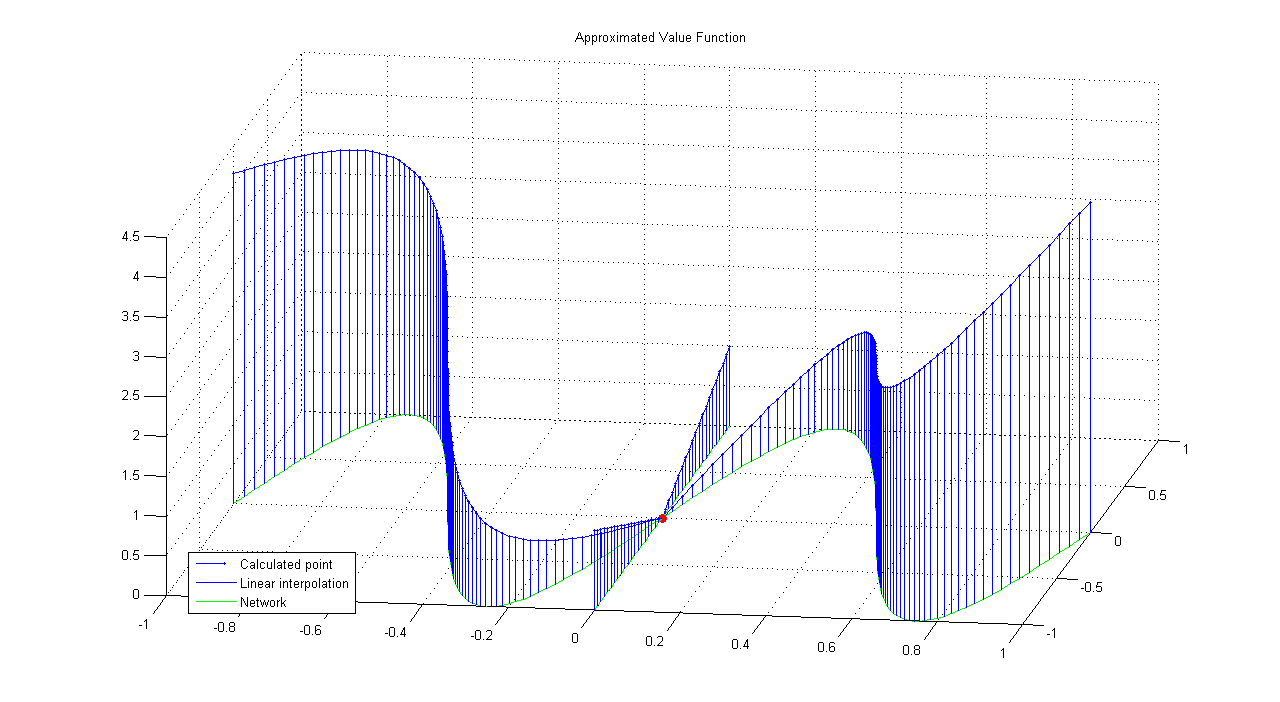}  \\
\caption{Test 1, $\Delta x=0.025$.} \label{fig2}
\end{center}
\end{figure}

In the first test we consider a five knots graph with two straight arcs and two sinusoidal ones (see figure \ref{scheme2}). The only boundary knot is the one placed at the origin and  the value of the solution at this knot is fixed to zero. The cost function is constant, i.e. $f(x)\equiv 1$  on  $\Gamma$. In this case the correct solution is
\begin{equation}
\begin{split}
u(x)=&dist(x,0)=|x_2| \quad \hbox{ for the straight arcs} \\
u(x)=&\int^{|x_1|}_0(\sqrt{1+(2 \pi \cos{2\pi t})})dt \quad \hbox{ for sinusoidal arcs}
\end{split}
\end{equation}
An approximated solution is shown in Figure \ref{fig2}. In Table \ref{tab2}, we compare the exact  solution with the approximate one, obtained by the scheme.  We observe a numerical convergence to the correct solution in $L_2$-norm and in the uniform one. As uniform norm we consider the maximum of the uniform norm of the error on every arc and as $L_2$-norm the maximum of the $L_2$-norm on every arc.  We can observe an order of convergence close to $0.5$ that is the typical theoretical order of convergence in the uniform norm of semi-Lagrangian schemes in $\R^n$, (see for instance \cite{ff}).\par

\begin{center}
\begin{table}[hb]
\begin{tabular}{||p{3.3cm}||*{4}{c|}|}
\hline
        $\Delta x=h $      &   $||\cdot||_\infty$  &$Ord(L_\infty)$&  $||\cdot||_2$ &  $Ord(L_2)$ \\
\hline
\hline
\bfseries 0.2 & 0.1468   &   &  0.1007 &  \\
\hline
\bfseries 0.1  &  0.0901 &  0.7043   & 0.0639  &  0.6562\\
\hline
\bfseries 0.05  &  0.0630   &    0.5162  &  0.0491 &  0.3801\\
\hline
\bfseries 0.025  &  0.0450 &   0.4854 &    0.0402 &  0.2885\\
\hline
\bfseries 0.0125  & 0.0321   &    0.4874  &  0.029   &  0.4711\\
\hline

\end{tabular}
\medskip
\caption{Test 1.}\label{tab2}
\end{table}
\end{center}

\begin{figure}[th]
\begin{center}
\includegraphics[height=6cm]{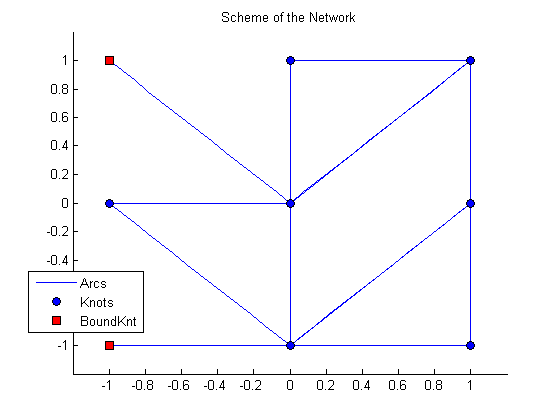}
\caption{Test 2, structure of the graph.} \label{scheme3}
\end{center}
\end{figure}

\clearpage

\begin{figure}[ht]
\begin{center}
\includegraphics[height=7cm]{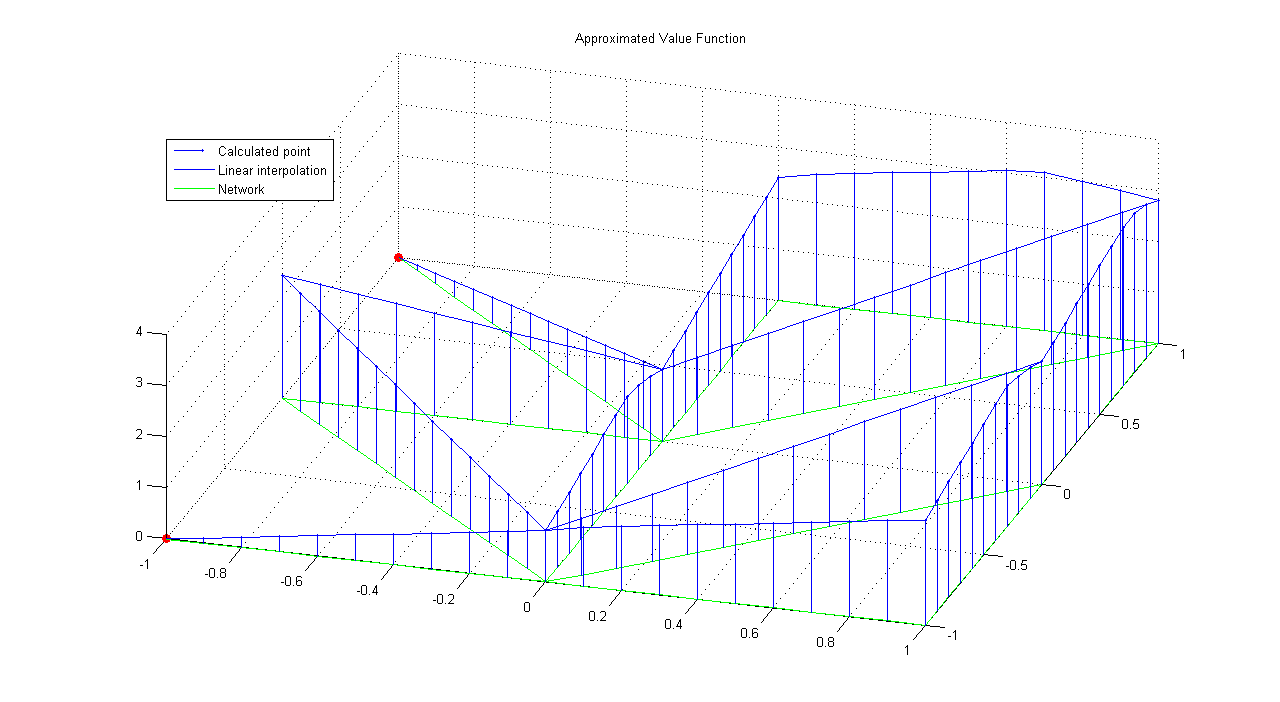}  \\
\caption{Test 2, $\Delta x=0.1$.} \label{fig3}
\end{center}
\end{figure}

In the second test we present a more complicated graph with  two boundary vertices and a   several connections among the arcs. Also in this case, we consider a constant cost function $f(x)\equiv 1$ on  $\Gamma$. In Table \ref{tab3} and in Figure \ref{fig3} we show our results.\\
In this case we observe an improvement of  order of convergence with respect to the previous example. This is due to the fact that the graph is compose  of only straight arcs and  this  reduces  the error due to the  piecewise linear discretization of the  arcs.\par

\begin{center}
\begin{table}[hb]
\begin{tabular}{||p{3.3cm}||*{4}{c|}|}
\hline
        $\Delta x=h $      &   $||\cdot||_\infty$  &$Ord(L_\infty)$&  $||\cdot||_2$ &  $Ord(L_2)$ \\
\hline
\hline
\bfseries 0.2 & 0.1716   &   &  0.0820 &  \\
\hline
\bfseries 0.1  &  0.0716 &  1.2610   & 0.0297 &  1.4652\\
\hline
\bfseries 0.05  &  0.0284   &    1.3341  &  0.0127&  1.2256\\
\hline
\bfseries 0.025  &  0.0126 &   1.1611 &    0.0072 &  0.8188\\
\hline
\bfseries 0.0125  &  0.0056 &   1.1699 &    0.0037 &  0.9605\\
\hline

\end{tabular}
\medskip
\caption{Test 2.}\label{tab3}
\end{table}
\end{center}

\begin{figure}[ht]
\begin{center}
\includegraphics[height=6cm]{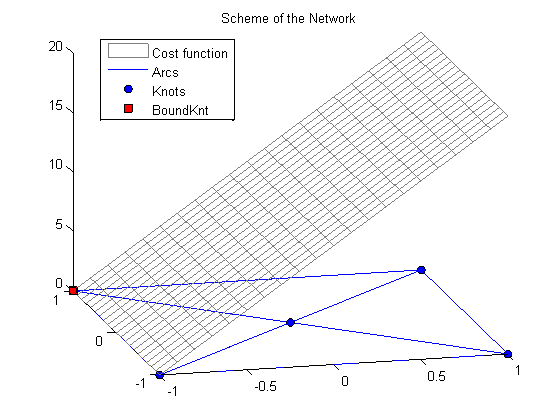}
\caption{Test 3, structure of the graph.} \label{scheme4}
\end{center}
\end{figure}

\begin{figure}[ht]
\begin{center}
\includegraphics[height=7cm]{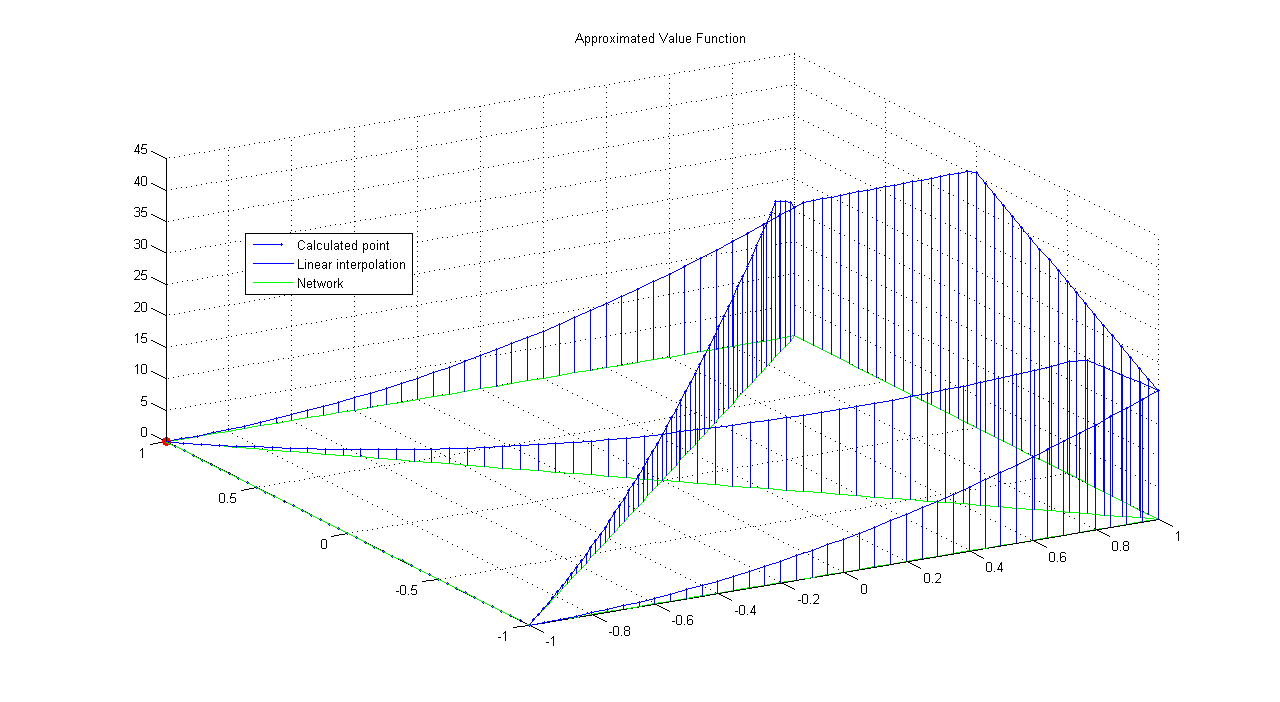}  \\
\caption{Test 3, $\Delta x=0.05$.} \label{fig4}
\end{center}
\end{figure}

\clearpage

In the third test we consider a  five knots graph (figure \ref{scheme4}), with  a running cost which is not constant. For any point on the graph $x=(x_1,x_2)\in \Gamma$, we take $f(x)=10(x_1-1)+\eta$, hence  $f(x)\geq\eta>0$  for  $x\in \Gamma$. In the example, we set $\eta=10^{-10}$. The graph of the approximate solution is shown in Figure \ref{fig4}. Also in this case we provide a experimental table of convergence for the error (Table \ref{tab4}). In absence of an exact solution we compare the approximation for various grid sizes with a discrete solution $U_{ex}$ on a fine grid ($\Delta x= 0.005$).

\begin{center}
\begin{table}[hb]
\begin{tabular}{||p{3.3cm}||*{4}{c|}|}
\hline
        $\Delta x=h $      &   $||\cdot||_\infty$  &$Ord(L_\infty)$&  $||\cdot||_2$ &  $Ord(L_2)$ \\
\hline
\hline
\bfseries 0.2 & 0.3800   &   &  0.2078 &  \\
\hline
\bfseries 0.1  &  0.1800 &  1.078   & 0.0855 &  1.2812\\
\hline
\bfseries 0.05  &  0.08   &    1.1699  &  0.0419&  1.029\\
\hline
\bfseries 0.025  &  0.035 &   1.1926 &    0.0222 &  0.9164\\
\hline
\bfseries 0.0125  &  0.0166 &   1.0762 &    0.0103 &  1.1079\\
\hline

\end{tabular}
\medskip
\caption{Test 3.}\label{tab4}
\end{table}
\end{center}

\begin{figure}[ht]
\begin{center}
\includegraphics[height=5.5cm]{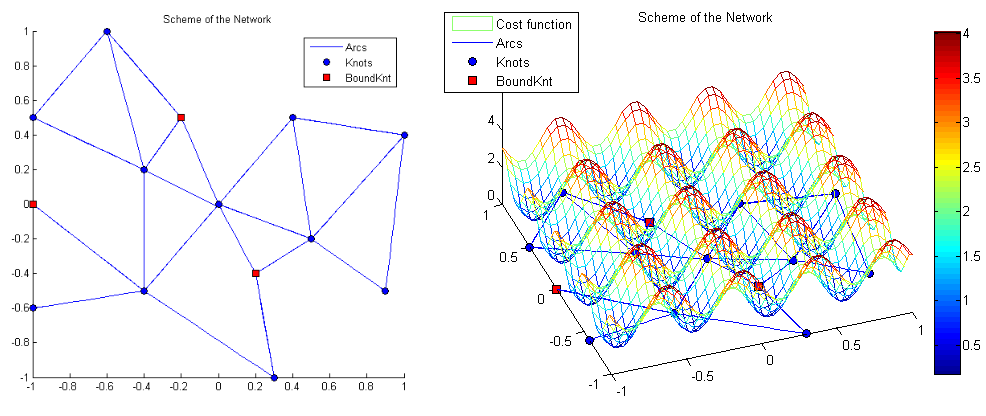}
\caption{Test 4, structure of the graph.} \label{scheme5}
\end{center}
\end{figure}

\begin{figure}[ht]
\begin{center}
\includegraphics[height=7cm]{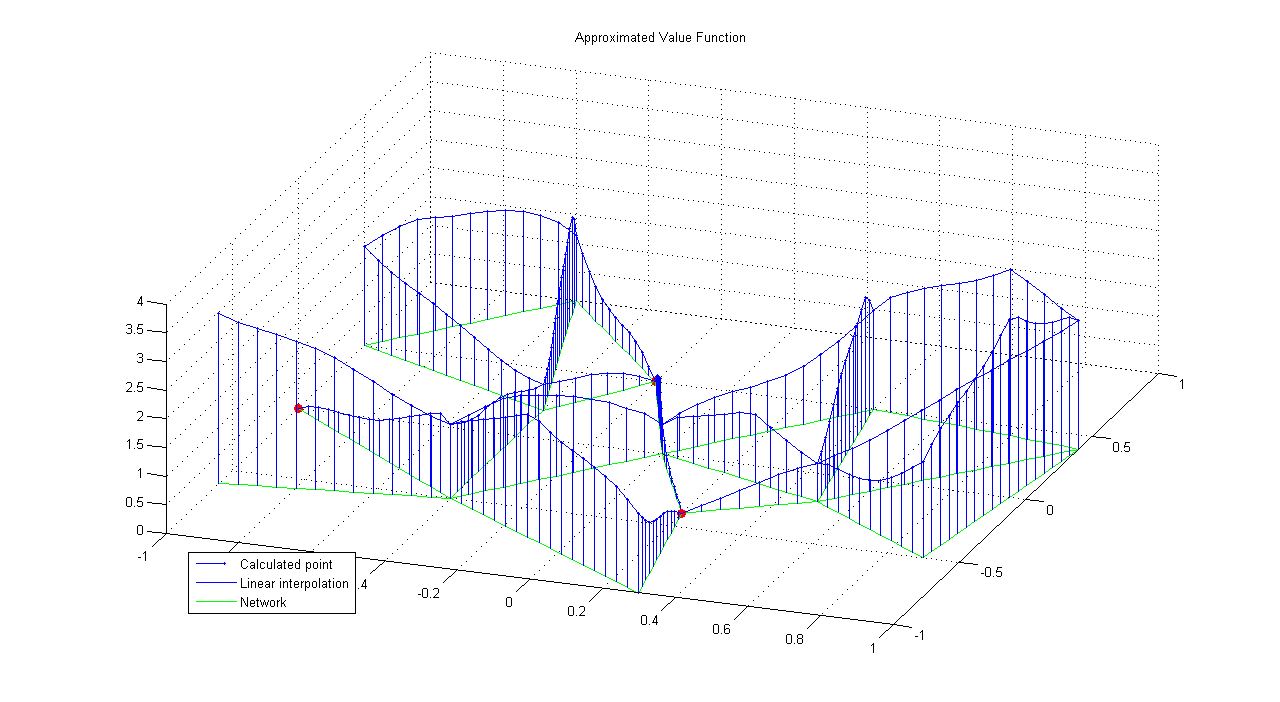}  \\
\caption{Test 4, $\Delta x=0.05$.} \label{fig5}
\end{center}
\end{figure}

As our last test we consider a graph with several boundary points and a more complicated running cost function $f$. A representation of this graph is shown in Figure \ref{scheme5}. We consider the following function $f$
\begin{equation}
  f(x_1,x_2)=2.1-\sin(4 \pi x_1)+\cos(4 \pi x_2)
\end{equation}
obviously, because of the regularity of this function, its restriction on the arcs of the graph is continuous.\\
In Table  \ref{tab5} we show a comparison for the error in various grid steps. Also in this case, in absence of the correct solution, we consider as correct the approximation on a fine grid ($\Delta x=0.005$).

\begin{center}
\begin{table}[hb]
\begin{tabular}{||p{3.3cm}||*{4}{c|}|}
\hline
        $\Delta x=h $      &   $||\cdot||_\infty$  &$Ord(L_\infty)$&  $||\cdot||_2$ &  $Ord(L_2)$ \\
\hline
\hline
\bfseries 0.2 & 0.7049   &   &  0.3676 &  \\
\hline
\bfseries 0.1  &  0.2925 &  1.2690   & 0.1557 &  1.2394\\
\hline
\bfseries 0.05  &  0.1460  &    1.0025  &  0.0777&  1.0028\\
\hline
\bfseries 0.025  &  0.0728 &   1.0040 &    0.0320 &  1.2798\\
\hline
\bfseries 0.0125  &  0.0375 &   0.9570 &    0.0108&  1.5670\\
\hline

\end{tabular}
\medskip
\caption{Test 4.}\label{tab5}
\end{table}
\end{center}

\bibliographystyle{model1b-num-names}
\bibliography{mybib}
\end{document}